\documentclass[12pt]{article}

\usepackage{amssymb, amsmath, amsthm, graphicx, comment}
\usepackage[left=1in,top=1in,right=1in]{geometry}
\usepackage[color=Orange!50!white, textwidth=22mm]{todonotes}

\usepackage{caption}
\usepackage{subcaption}
\usepackage[percent]{overpic}
\usepackage{pgfplots}
\pgfplotsset{compat=1.15}
\usepackage{mathrsfs}
\usetikzlibrary{arrows}
\usepackage{tikz}
\usetikzlibrary{calc}

\theoremstyle{plain}
      \newtheorem{theorem}{Theorem}
      \newtheorem{lemma}[theorem]{Lemma}

      \newtheorem{problem}[theorem]{Problem}
      
      \newtheorem{corollary}[theorem]{Corollary}

\theoremstyle{definition}
      \newtheorem{definition}{Definition}

\theoremstyle{remark}

	\newcommand{\ZZ}{{\mathbb Z}}

\def\diam{\mathrm{diam}}

\long\def\comment#1{}

\title{Cutting a convex body into fat parts and approximating  Euclidean distance by graph distances }
\author{J\'anos Pach\thanks{Alfr\'ed R\'enyi Institute of Mathematics, Budapest, Hungary. Research partially supported by ERC Advanced Grant `GeoScape.' Email: \texttt{pach@renyi.hu}.}
\and G\'abor Tardos\thanks{Alfr\'ed R\'enyi Institute of Mathematics, Budapest, Hungary. Research partially supported by ERC Advanced Grant `ERMiD.' Email: \texttt{tardos@renyi.hu}.}}

\begin{document}
\date{}

\maketitle

\begin{abstract}
Can one construct a graph $G$ on the set of integer points $\ZZ^2$ in the plane such that the length of the shortest path between any two vertices of $G$ differs from their Euclidean distance by at most an absolute constant? This question of Benjamini, Erd\H os, Kleiner, Kozma, Schramm, and the first-named author has been open for a long time. We give an affirmative answer to a weaker form of this question, based on the following geometric statement, which is of independent interest. There exists a constant $c>0$ such that for every $i=1,2,\ldots,$ every $\rho$-fat plane convex set $S$ can be cut into $2^i$ convex pieces of equal area, each of which is at least $c\rho$-fat. (A convex set is \emph{$\rho$-fat} if the ratio of its inradius to its circumradius is at least $\rho$.)

We prove that there exists an (unweighted) spanning subgraph $G$ of an enlarged copy of $\ZZ^2$ such that, for every pair of vertices at Euclidean distance $d$, their shortest-path distance in $G$ lies between $d-O(1)$ and $d+o(d^{5/6})$. The same bound can be achieved by a planar graph with vertex set $\ZZ^2$, in which every edge joins two vertices at Euclidean distance at most 2.
\end{abstract}

\section{Introduction}

In the mathematical and computer science literature, there are several competing quantities to measure how ``round'' a convex body $S$ is: fatness~\cite{vdSHO93, vdSO94, deBKV02, Gr08}, chunkiness~\cite{Ci78}, aspect ratio~\cite{Sh14}, eccentricity~\cite{Le59}, etc. For a comparison of several of these parameters, see also~\cite{deBKV02,Gr08}.

In Section~\ref{sec2}, we introduce and analyze a new measure. We use $d(p,q)$ to denote the Euclidean distance between $p$ and $q$.

\begin{definition}\label{roundness}
Let $S$ be a compact convex set in the plane consisting of more than one point. For a pair of distinct points $p,q\in S$, let $rs$ be the segment that is the intersection of $S$ with the orthogonal bisector of the segment $pq$, and let $C(S,p,q):=\frac{d(r,s)}{d(p,q)}$.

The \emph{roundness} of $S$ is defined as
$$C(S):=\inf_{p,q}C(S,p,q).$$
\end{definition}

The roundness of a disk is 1, and the roundness of a square is also 1. Observe that this parameter is closely related to a classical measure: the ratio of the inradius $r(S)$ to the circumradius $R(S)$. It is not hard to prove that for any plane convex body $S$, we have $0.28C(S)\le r(S)/R(S)\le2C(S)$. In particular, the more elongated $S$ is, the lower its roundness.

It seems plausible that a ``round'' convex body in the plane can be cut into many convex pieces of equal area that are comparably round. Yet the proof of this fact is far from obvious.

We start Section~\ref{sec3} by establishing the following result.

\begin{theorem}\label{fat}
Given any natural number $k$, every plane convex body $S$ can be cut into $2^k$ convex pieces of equal area whose roundness is at least $\min(C(S),1/12)$.
\end{theorem}

Given a graph $H$ and two vertices $u,v\in V(H)$, let $d_H(u,v)$ denote their \emph{graph distance}, that is, the minimum number of edges along a path connecting $u$ and $v$ in $H$. We will apply Theorem~\ref{fat} and its proof to study an old problem in metric geometry: Can one define a graph $G$ whose vertex set is ${\ZZ}^2$, the set of integer points in the plane, and for which the graph distance and the Euclidean distance are almost the same? Some historical remarks can be found in the second half of this introduction.

At the end of Section~\ref{sec3}, based on the recursive proof of Theorem~\ref{fat}, we will build a finite connected subdivision graph $H$ whose vertices lie in a bounded region of the plane.
\smallskip

In Section~\ref{sec4}, we will compare the graph distance and the Euclidean distance in $H$. In Section~\ref{sec5}, using a limiting argument, we construct a single infinite graph $H$, for which the graph distance and the Euclidean distance are asymptotically the same.

\begin{theorem}\label{generalizedpinwheel}
There exists an infinite set of points $P$ in the plane and a graph $H$ with $V(H)=P$ that satisfy the following properties.
\begin{enumerate}
\item The distance between any two distinct vertices of $P$ is at least $1/3$.
\item For any point $x$ of the plane, there is a vertex $p\in P$ within distance at most $1/3$.
\item For any two vertices $p,q\in P$, we have
$$d(p,q)\le d_H(p,q)=d(p,q)+O\left(\frac{d(p,q)}{\log^{2/3}(d(p,q))}\right).$$
\end{enumerate}
\end{theorem}

Theorem~\ref{generalizedpinwheel} easily implies the following.

\begin{corollary}\label{inwheelongrid}
There exists a graph $H'$ on the vertex set ${\ZZ}^2$ such that for any two integer points $p$ and $q$, we have $$d_{H'}(p,q)=d(p,q)+O\left(\frac{d(p,q)}{\log^{2/3}(d(p,q))}\right),$$
as $d(p,q)$ tends to infinity.
\end{corollary}

\begin{proof}[Proof of Corollary~\ref{inwheelongrid}]
To construct $H'$, we take a point set $P$ described in the theorem, blow it up by a sufficiently large constant factor, and then embed it in the grid. For a precise argument, we first bound the maximum degree of the graph $H$ from the theorem. The neighbors of a vertex $x\in P$ in the graph $H$ are all within distance $1$ from $x$, but are at distance at least $1/3$ from each other. Therefore, the open disks of radius $1/6$ centered at the neighbors of $x$ are pairwise disjoint and are contained in the disk of radius $7/6$ centered at $x$. Comparing areas, we obtain
$$
\deg_H(x)<\frac{\pi(7/6)^2}{\pi(1/6)^2}=49.
$$
Thus, the maximum degree of $H$ is at most $48$.

To obtain $H''$, replace each edge of $H$ by a path of length $501$. The new graph $H''$ contains the vertices in $P$ and also new degree $2$ vertices on the connecting paths. Every vertex $v$ of $H''$ is within $H''$-distance at most $250$ from a unique vertex $x\in P$.

Now we injectively map the vertices $v$ of $H''$ to points $f(v)$ of the grid $\ZZ^2$ in such a way that, if $x\in P$ is the unique vertex satisfying
$d_{H''}(v,x)\leq 250,$ then $d(f(v),501x)<501/6.$ This is possible because

(i) every vertex $v$ of $H''$ is associated with exactly one vertex $x\in P$;

(ii) there are at most $48\cdot250+1=12001$
vertices $v$ associated with any fixed point $x\in P$;

(iii) the open disks of radius $501/6$ around the points in $501P$ are pairwise disjoint, since distinct points of $P$ are at distance at least $1/3$ from each other; and

(iv) every open disk of radius $501/6$ contains more than $12001$ grid points. Indeed, such a disk contains an open axis-parallel square of side length $118$, because $59\sqrt{2}<501/6.$

Every open interval of length $118$ contains at least $117$ integers, so this square contains at least
$117^2=13689>12001$ points of $\ZZ^2$.

Finally, we construct $H'$ on the vertex set $\ZZ^2$ by connecting $f(v)$ and $f(w)$ for every edge $vw$ of $H''$, and adding a single additional edge for every grid point $z$ \emph{not} of the form $f(v)$. For such a point $z\in\ZZ^2$, we select $x\in P$ within distance $1/3$ of $z/501$ and add the edge $zf(x)$. The vertices of $\ZZ^2$ that are not of the form $f(v)$ have degree $1$, and therefore they cannot create shortcuts between vertices in $f(V(H''))$. It follows that, for any $x,y\in P$,
$$
d_{H'}(f(x),f(y))=501d_H(x,y).
$$
Furthermore,
$$
d(f(x),f(y))=501d(x,y)+O(1).
$$
For every $z\in\ZZ^2$, there exists $x\in P$ such that
$d(z,f(x))=O(1)$ and $d_{H'}(z,f(x))=O(1).$ Indeed, if $z$ is not of the form $f(v)$, this follows from the definition of the additional edge incident to $z$. If $z=f(v)$, take the unique $x\in P$ satisfying $d_{H''}(v,x)\leq250$.

Using part~3 of the theorem and increasing the implicit constant, if necessary, to cover bounded distances, we obtain, for $d(z,z')\geq 2$,
$$d_{H'}(z,z')=d(z,z')+O\left(\frac{d(z,z')}{\log^{2/3}(d(z,z'))}\right).$$
This proves the corollary.
\end{proof}

The order of magnitude in the error term is not optimal. In the last section of this paper, we modify an idea of Cai \emph{et al.}~\cite{CaC26} concerning a closely related problem, the so-called ``Squishy Grid Problem," to obtain the following.

\begin{theorem}\label{best}
There is a spanning subgraph $H$ of an enlarged square grid (in particular, $V(H)=\frac32{\ZZ}^2$, $E(H)\subset\left\{\{x,y\}\in\binom{V(H)}2\mid d(x,y)=3/2\right\}$) such that for any two vertices $p$ and $q$, we have
$$d(p,q)-O(1)\le d_H(p,q)=d(p,q)+o(d(p,q)^{5/6}).$$
\end{theorem}

\begin{corollary}
There is a planar graph $H'$ on the vertex set $\ZZ^2$ with all edges connecting vertices of distance at most $2$, such that for any two vertices $p$ and $q$, we have
$$d(p,q)-O(1)\le d_H(p,q)=d(p,q)+o(d(p,q)^{5/6}).$$
\end{corollary}

\begin{proof}
We use the same idea as in the proof of Corollary~\ref{inwheelongrid}. However, since  the graph $H$ provided Theorem~\ref{best} has a particularly simple structure, the transformation is simpler here. 

We use two types of edges in $H'$. First, for every $p\in\ZZ^2$, we connect $3p$ to every grid point at Euclidean distance at most $\sqrt2$ from $3p$. Next, for each edge $pq$ of $H$, we add one additional edge. If $p=(\frac32x,\frac32y)$, $q=(\frac32(x+1),\frac32y)$, then we add the edge joining $(3x,3y)$ to $(3x+2,3y)$. Similarly, if $p=(\frac32x,\frac32y)$, $q=(\frac32x,\frac32(y+1))$, then we add the edge joining $(3x,3y)$ to $(3x,3y+2)$.

The claimed bounds on graph distances follow from the corresponding bounds for $H$. Indeed, $H'$ can be viewed as follows: first enlarge the vertex set of $H$ by a factor of $2$, then subdivide each edge, with the subdivision vertex placed at a trisection point of the enlarged edge, and finally add short edges connecting the remaining grid points to this enlarged and subdivided copy of $H$. 

One can get a planar drawing of $H'$ by drawing every edge of length less than $2$ as a straight-line segment. An edge of length 2 can be drawn arbitrarily close to the corresponding straight-line segment, with a small detour around an intermediate grid point. These detours can be chosen sufficiently small that no crossings are introduced. Hence $H'$ is planar.
\end{proof}
\smallskip

\textbf{Historical remarks, background.}
A classical problem in computational geometry is the following. Given a finite set of points, $P$, in the plane, we would like to triangulate it by connecting certain pairs of points of $P$ by straight-line edges in such a way that the length of the shortest path between any two points in $P$, using these edges, is not much longer than their Euclidean distance. (Here the length of a path is the sum of the lengths of its edges.) Chew~\cite{Ch86} discovered that the Delaunay triangulation achieves a constant factor approximation, and applied this result to several motion planning and visibility problems. The value of the constant has been subsequently improved~\cite{DoFS90, KeG92}. The case when one is allowed to add a limited number of so-called ``Steiner points'' to $P$ turned out to be important in mesh generation~\cite{BeSA17}.
\smallskip

In the infinite setting, a very similar problem was raised in metric geometry. A \emph{$\delta$-net} in a metric space is a subset $P$ with the property that every point of the space is within a distance of $\delta$ from at least one element of $P$.
A \emph{edge-weighted}, bounded degree graph whose vertex set is an $\delta$-net for some $\delta>0$, is said to be \emph{uniform} if the weights (``lengths'') of its edges are between two positive constants. It is a basic question in metric geometry to decide which complete manifolds are approximable, up to an additive error, by a uniform graph $H$. In other words, when does there exist an absolute constant $C$ such that, for any pair of vertices, the minimum total length of a path in $H$ connecting them differs from their distance by at most $C$. For the Euclidean plane, Burago and Ivanov~\cite{BuI15} constructed such a uniform graph $H$ on the vertex set $V(H)=\mathbb{Z}^2$.
\smallskip

Due to its applications to network-design problems and distributed computing, the same problem has also been widely studied in a discrete abstract setting. Given a possibly weighted, connected graph $G$, we want to find a so-called ``spanner'' of $G$, i.e., a subgraph $H\subseteq G$ with $V(H)=V(G)$ such that the distance between any two vertices is roughly the same in $H$ as in $G$; see \cite{PeS89, AhBD20}. In the present paper, unless we state it otherwise, \emph{every edge has weight (``length'') $1$}, so that the distance $d_G(u,v)$ between any pair of vertices $u,v\in V(G)$ is equal to the smallest number of edges along a path connecting them. Of course, if we choose $H$ to be $G$, then $d_H(u,v)=d_G(u,v)$ for every $u$ and $v$. The goal is to find a sparse subgraph of $G$ such that the \emph{multiplicative distortion} $$\max_{u,v\in V}\frac{d_H(u,v)}{d_G(u,v)}$$ or the \emph{additive distortion}
$$\max_{u,v\in V}\left(d_H(u,v)-d_G(u,v)\right)$$
remain as small as possible.
\smallskip

Motivated by the above considerations, several researchers, including Bruce Kleiner, Gady Kozma, Oded Schramm, Itai Benjamini, Paul Erd\H os and the first named author, independently, raised the following question, which is still open.

\begin{problem}\label{additive}
Does there exist an unweighted graph $H$ on the vertex set $V(H)={\mathbb{Z}}^2$ such that
\[\sup_{u,v\in {\mathbb{Z}}^2}\left(d_H(u,v)-d(u,v)\right)<\infty\;?\]
\end{problem}

Using the terminology of Benjamini~\cite{Be13}, a graph $H$ with the above property is \emph{slack-isometric} to the Euclidean plane. A special case of this problem was discussed in~\cite{Be26}.
\smallskip

In~\cite{PaPS90}, a weaker statement was established: \emph{For every $\varepsilon>0$, there exists a graph $H=H_{\varepsilon}$ with $V(H)={\mathbb{Z}}^2$ such that
\[1-\varepsilon<\frac{d_H(u,v)}{d(u,v)}<1+\varepsilon,\]
for any pair of vertices $u,v\in {\mathbb{Z}}^2$ that are sufficiently far from each other.}
\smallskip

Radin and Sadun~\cite{RaS96} used the edge set of the so-called pinwheel tiling~\cite{Ra94, Ra95} (see also~\cite{CoR98}) to construct a single graph $H$ that satisfies the above properties for \emph{all} $\varepsilon>0$ at the same time. Their analysis is based on a compactness argument; it provides no information on how fast the ratio of the graph distance and the Euclidean distance tends to $1$, as $d(u,v)\rightarrow\infty$, but the paper \cite{CaC26} proved $d_H(u,v)=d(u,v)(1+O(\log^{-c}\log d(u,v)))$ for some absolute constant $c>0$.
\smallskip

Theorem~\ref{generalizedpinwheel} gives a different construction, based on partitions into round convex pieces, with an explicit logarithmic error term. Its proof was inspired by the pinwheel construction. Theorem~\ref{best} adapts the highway construction of Cai \emph{et al.} to obtain an unweighted graph on an enlarged grid with polynomial additive error.

\section{Measures of fatness}\label{sec2}

In this section, we explore \emph{roundness}, our novel measure of fatness for compact convex planar bodies. While it is closely related to more standard measures as shown by Lemma~\ref{compare}, it is uniquely suitable to our purposes as a planar body can be cut into many parts of prescribed area typically without \emph{any} decrease in the roundness. Strictly speaking, this does not always hold. For example, the roundness of a circular disk is $1$ (the maximal value of roundness), but if we cut it to two half-disks, the roundness of those are only $1/2$. But as we will see in Lemma~\ref{sections}, it holds generally enough. Lemma~\ref{sections} will then be used in the next section to prove Theorem~\ref{fat}. The proof of Theorem~\ref{generalizedpinwheel} is somewhat more complex: there we need partitions with separating segments of predetermined directions. When cutting in a predetermined direction the roundness will somewhat decrease. To counteract this, we also need to be able to cut a fat body in such a way that the roundness never decreases but it \emph{significantly increases} for certain parts. The last part of Lemma~\ref{sections} facilitates this, while Lemma~\ref{longcut} establishes technical properties of the cuts in predetermined directions that we will also need for the proof of Theorem~\ref{generalizedpinwheel}.

Let $\mu(S)$ stand for the area of a (typically compact convex) part $S$ of the plane and let $\diam(S)$ stand for its diameter.

We compare roundness to a more standard measure of fatness defined as $D(S):=\mu(S)/(\diam(S))^2$ for a compact convex set $S$ in the plane with more than a single point. The ratio of the inradius $r(S)$ to the circumradius $R(S)$ of $S$ is another standard measure of fatness. We will not use this measure here, but note that it is also closely related to $D(S)$: for example,
$$0.56D(S)<\frac{r(S)}{R(S)}<2D(S).$$

In our first lemma, we state some simple properties of roundness, including its close connection to $D$.

\begin{lemma}\label{compare}
Let $S$ be a compact convex set in the plane with more than a single point.
\begin{enumerate}
\item$0\le C(S)\le1$.
\item$D(S)\le C(S)\le2D(S)$.
\item The infimum in the definition of $C(S)$ is a minimum, namely $C(S)=C(S,p,q)$ for a pair of distinct points on the boundary of $S$.
\end{enumerate}
\end{lemma}

\begin{proof}
$C(S)\ge0$ is trivial. Throughout this proof, intersections of a convex set with a line are regarded as possibly degenerate segments.

Consider $p,q\in S$ with $d(p,q)=\diam(S)$. Let the segment $rs$ be the intersection of $S$ and the orthogonal bisector of $pq$, see Figure~\ref{fig:diameter-kite}. As $d(r,s)\le\diam(S)$, we have $C(S)\le C(S,p,q)=\frac{d(r,s)}{d(p,q)}\le1$ proving part~1. Also, the area of $S$ is at least the area of the kite $prqs$, which is $d(p,q)d(r,s)/2=C(S,p,q)\diam(S)^2/2$. Thus $D(S)=\mu(S)/\diam(S)^2\ge C(S,p,q)/2\ge C(S)/2$ proving the second inequality in part~2.

For the first inequality in part~2, consider an arbitrary pair of distinct points $p,q\in S$ and let $rs$ be the intersection of $S$ and the orthogonal bisector of $pq$. Notice that $S$ is contained within a (possibly degenerate) trapezoid $XYZT$ with all sides touching $S$, $XY$ at $r$, $ZT$ at $s$ and the other two sides orthogonal to $pq$, see Figure~\ref{fig:supporting-trapezoid}. Let $d_1$ and $d_2$ be the distances between the parallel lines $XT$ and $rs$, and between $YZ$ and $rs$, respectively. Clearly, $d_1,d_2\ge d(p,q)/2$ and $d_1+d_2\le\diam(S)$. We also assume, without loss of generality, that $d_1\le d_2$. We have $d(X,T)=d(r,s)-\alpha d_1$ and $d(Y,Z)=d(r,s)+\alpha d_2$ for the same value $\alpha$ (determined by the angles of the trapezoid), thus $0\le d(X,T)$ implies $\alpha\le d(r,s)/d_1\le d(r,s)/(d(p,q)/2)=2C(S,p,q)$. The area of the trapezoid is
\begin{eqnarray*}
(d_1+d_2)(d(X,T)+d(Y,Z))/2&=&(d_1+d_2)(2d(r,s)+\alpha(d_2-d_1))/2\\
&\le&\diam(S)(2d(r,s)+2C(S,p,q)(d_2-d_1))/2\\
&=&\diam(S)(d(r,s)+C(S,p,q)(d_1+d_2-2d_1))\\
&\le&\diam(S)(d(r,s)+C(S,p,q)\diam(S)-2C(S,p,q)(d(p,q)/2)\\
&=&C(S,p,q)\diam(S)^2.
\end{eqnarray*}
The area of $S$ is bounded by the area of the trapezoid, thus $D(S)=\mu(S)/\diam(S)^2\le C(S,p,q)$. As $p$ and $q$ were arbitrary, we also have $D(S)\le C(S)$ proving the second part of the lemma.

\begin{figure}[!htbp]
\centering
\begin{minipage}[t]{0.49\textwidth}
\vspace{0pt}\centering
\includegraphics[width=\linewidth]{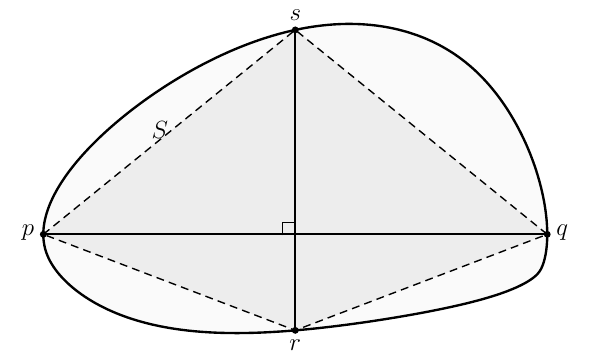}
\caption{}\label{fig:diameter-kite}
\end{minipage}\hfill
\begin{minipage}[t]{0.49\textwidth}
\vspace{0pt}\centering
\includegraphics[width=\linewidth]{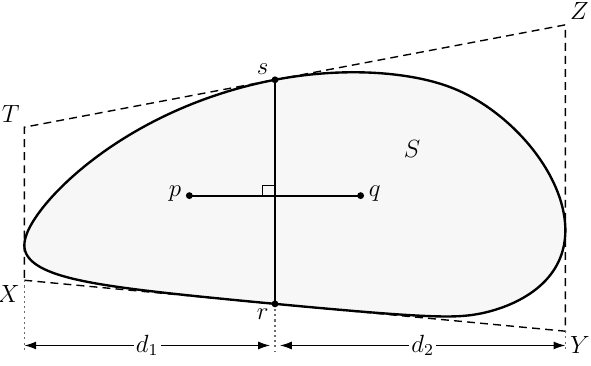}
\caption{}\label{fig:supporting-trapezoid}
\end{minipage}
\end{figure}

For part~3, consider the distinct points $p,q\in S$ and let $q'$ be an intermediate point on the segment $pq$. Let the segment $rs$ be the intersection of $S$ with the orthogonal bisector of $pq$ and similarly, let $r's'$ be the intersection of $S$ with the orthogonal bisector of $pq'$, see Figure~\ref{fig:nested-bisectors}. As the triangle $prs$ is contained in the convex set $S$, we have $d(r',s')\ge d(r,s)\frac{d(p,q')}{d(p,q)}$. Therefore, $C(S,p,q)\le C(S,p,q')$. This implies that in the infimum in the definition of $C(S)$, it is enough to consider points $p$, $q$ on the boundary of $S$, because extending the $pq$ interval in either direction does not increase the value $C(S,p,q)$. It is easy to see that $C(S,p,q)$ is continuous in $p$ and $q$, so by the compactness of the boundary of $S$ we are almost done with proving the third claim in the lemma. We only have to consider the case when the infimum in the definition is approached by point pairs $p_i\ne q_i$ with both sequences $p_i$ and $q_i$ tending to the same point $p$. In this case $\lim_iC(S,p_i,q_i)=\infty$ unless $S$ is a segment or $p$ is corner point of $S$, that is, there is more than one tangent line touching $S$ at $p$. In the first case, we simply have $C(S,p_i,q_i)=0$, so all point pairs give the minimum. In the latter case, take the angular bisector of $S$ at $p$ (bisecting the angle between the extreme tangents) and let $q$ be the other intersection of this line with the boundary of $S$. It is not hard to see that $C(S,p,q)\le\lim_i C(S,p_i,q_i)$.

\begin{figure}[!htbp]
\centering
\includegraphics[width=0.62\textwidth]{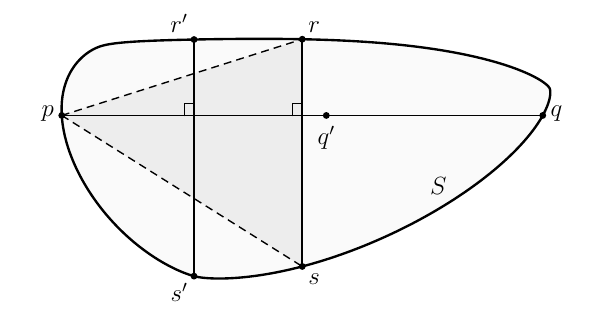}
\caption{}\label{fig:nested-bisectors}
\end{figure}
\end{proof}

We remark, that while $C(S)=0$ implies that $D(S)=0$, hence $\mu(S)=0$ and so $S$ is a segment, $C(S)=1$ holds for varying shapes, including a circular disk and a square.

The next lemma is the key to the proof of our Theorem~\ref{fat}. In fact, part~1 of the lemma is enough for this proof. Part~2 is similar, but allows showing not only that the roundness can be preserved by the cutting procedure but also that in some parts it increases significantly. We will use it to obtain a partition that preserves roundness and has cuts in many arbitrarily chosen directions. To maintain compactness, we define the region ``between'' two parallel lines to include those lines.

\begin{lemma}\label{sections}
Let $S$ be a compact convex set in the plane with positive area. Let $x,y\in S$ be points with $d(x,y)=\diam(S)$. Let $l_1$ and $l_2$ be two lines perpendicular to $xy$ and $S'$ the part of $S$ between these lines. 
\begin{enumerate}
\item If $S'$ has more than a single point, then
$$C(S')\ge\min\left(C(S),\frac{\mu(S')}{6\mu(S)}\right).$$
\item If $S'$ contains the midpoint of the segment $xy$, then
$$C(S')\ge\min\left(\left(1-4\frac{\mu(S')}{\mu(S)}\right)\frac{C(S)}{2\frac{\mu(S')}{\mu(S)}+C(S))},\frac{\mu(S')}{6\mu(S)}\right).$$
\end{enumerate}
\end{lemma}

\begin{proof}
We need to show the desired lower bounds for $C(S',p,q)$ for any pair of distinct points $p,q\in S'$. Let $l$ be the perpendicular bisector of the segment $pq$ and let $l\cap S'$ be the segment $rs$.

For simple reference, we orient the plane so that $xy$ is horizontal with $x$ to the left of $y$ and assume $l_1$ is to the left of $l_2$. Let $d_0$ stand for the distance between $l_1$ and $l_2$. We start with two simple observations.

(i) $d_0\ge\frac{\mu(S')}{2\mu(S)}\diam(S)$. 

\noindent To see this, let $l_0$ be the vertical line with the length $d'$ of the interval $l_0\cap S'$ maximal. The claimed bound follows from $\mu(S')\le d_0d'$ and $\mu(S)\ge\diam(S)d'/2$, where the last formula gives the area of the convex hull of $\{x,y\}\cup(l_0\cap S')$.

(ii) $C(S',p,q)\ge\measuredangle rps/2$.

\noindent Let $t$ be the midpoint of the segment $pq$. The point $t$ is on the segment $rs$ orthogonal to $pq$. The claim follows from $$d(r,s)=d(r,t)+d(t,s)=$$ $$d(p,t)\tan\measuredangle rpt+d(p,t)\tan\measuredangle tps\ge d(p,t)(\measuredangle rpt+\measuredangle tps)=d(p,q)\measuredangle rps/2.$$

Note that the boundary of $S$ consists of two curves connecting $x$ and $y$, the \emph{top arc} above the \emph{bottom arc}. We distinguish several cases according to where the points $r$ and $s$ are on the boundary of $S'$, which consists of parts of $l_1$, $l_2$, the top arc, and bottom arc.
\smallskip

\noindent{\bf Case 1.} Neither $r$ nor $s$ is on the boundary of $S$. Clearly, they have to be on distinct lines $l_i$, so $d(r,s)\ge d_0$. We use (i) above to calculate
$$C(S',p,q)=\frac{d(r,s)}{d(p,q)}\ge\frac{d_0}{\diam(S)}\ge\frac{\mu(S')}{2\mu(S)},$$
a sufficient bound for both parts of the lemma.
\smallskip

\noindent{\bf Case 2.} One of $r$ and $s$ is on the boundary of $S$, the other is not. Assume without loss of generality that $r\in l_1$ and $s$ is on the top arc of the boundary of $S$. Also assume, still without loss of generality, that $p$ is above the line $rs$, see Figure~\ref{fig:section-case-two}.

Let $l_3$ be the vertical line through $y$. Consider the triangle $xyp$: as $d(x,p)\le\diam(S)=d(x,y)$, we must have $\gamma:=\measuredangle xyp\le\measuredangle xpy$. This implies by elementary calculation that the distance between $p$ and $l_3$ is at most $2\diam(S)\cos^2\gamma$. Using the facts that $p$ is above the line $rs$ and its projection $t$ to this line is at or to the left of $s$, furthermore, as $s$ is on the top arc of the boundary of $S$ it cannot be below the segment $pt$, we conclude that $s$ is to the right of $p$. The distance between $s$ and $l_3$ is at most the distance between $p$ and $l_3$, so at most $2\diam(S)\cos^2\gamma$. As $d(x,y)=\diam(S)$, all of $S$ must be on or to the left of $l_3$, so we can assume without loss of generality that $l_2$ is between $l_1$ and $l_3$. The distance of $r$ from $l_3$ is the distance of $l_1$ from $l_3$, at least the distance $d_0$ of $l_1$ from $l_2$, so at least $\frac{\mu(S')}{2\mu(S)}\diam(S)$ by (i) above. Thus, $d(r,s)\ge\diam(S)\mu(S')/(2\mu(S))-2\diam(S)\cos^2\gamma$ and $C(S',p,q)=d(r,s)/d(p,q)\ge\mu(S')/(2\mu(S))-2\cos^2\gamma$.

From (ii) above, we also have $C(S',p,q)\ge\alpha/2$ for $\alpha=\measuredangle rps$. Elementary calculations show that $\alpha+\gamma\ge\pi/2$, so $\cos^2\gamma=\sin^2(\pi/2-\gamma)\le\alpha^2$. Combining these estimates, we obtain
$$C(S',p,q)\ge\max\left(\alpha/2,\frac{\mu(S')}{2\mu(S)}-2\alpha^2\right)\ge\frac{\mu(S')}{6\mu(S)},$$
a satisfactory bound for both parts of the lemma.
\smallskip

\noindent{\bf Case 3.} $r$ and $s$ are on the same arc of the boundary of $S$ connecting $x$ and $y$. Assume without loss of generality that they are on the top arc and $p$ is above the segment $rs$, see Figure~\ref{fig:section-case-three}. $C(S',p,q)\ge\measuredangle rps/2$ by (ii) above. But clearly, $\measuredangle rps\ge \measuredangle xpy\ge\pi/3$ as $xy$ is a longest edge of the triangle $pxy$. So in this case $C(S',p,q)\ge\pi/6>\mu(S')/(6\mu(S))$, a satisfactory bound for both parts of the lemma.
\smallskip

\begin{figure}[!htbp]
\centering
\begin{minipage}[t]{0.49\textwidth}
\vspace{0pt}\centering
\includegraphics[width=\linewidth]{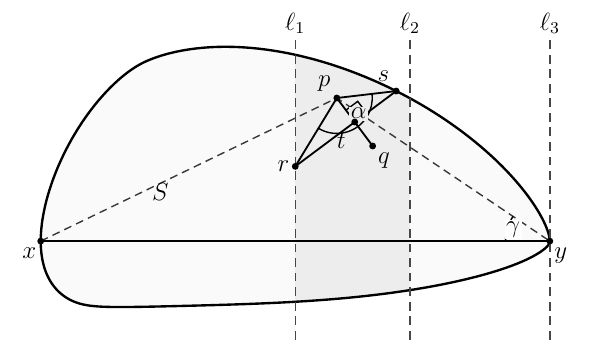}
\caption{}\label{fig:section-case-two}
\end{minipage}\hfill
\begin{minipage}[t]{0.49\textwidth}
\vspace{0pt}\centering
\includegraphics[width=\linewidth]{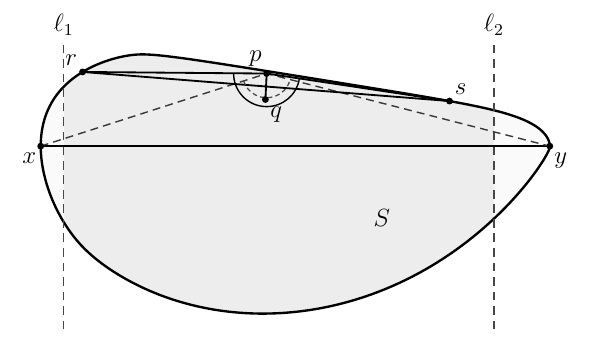}
\caption{}\label{fig:section-case-three}
\end{minipage}
\end{figure}

\noindent{\bf Case 4.} One of $r$ and $s$, is on the top arc of the boundary of $S$ connecting $x$ and $y$ and the other is on the bottom arc. In this case, $l\cap S'=l\cap S$, therefore $C(S',p,q)=C(S,p,q)\ge C(S)$, finishing the proof of part~1 of the lemma. (This argument also proves part~1 of the lemma in Case~3.) In the rest of this proof we concentrate on part~2.

We assume without loss of generality that $r$ is on the top arc and $s$ is on the bottom arc. We also assume that the distance $d_0$ between $l_1$ and $l_2$ is at most $\diam(S)/2$. We can do this as otherwise $\mu(S')/\mu(S)\ge1/4$ (we will see below that even $d_0=\diam(S)/2$ implies this), and so the first term in the minimum in part~2 of the lemma is negative or zero, so there is nothing to prove.

We claim that
$$\frac{\mu(S')}{\mu(S)}\ge\frac{d_0}{2\diam(S)}.$$
To see this, let $r_0s_0$ be the intersection of the orthogonal bisector of $xy$ and $S$. Note that $r_0s_0$ is contained in $S'$ by the assumption in part~2 of the lemma. $S$ is contained in a (possibly degenerate) trapezoid whose sides are the vertical lines through $x$ and $y$ and tangents of $S$ through $r_0$ and $s_0$. The area of this trapezoid is $\diam(S)d(r_0,s_0)$. On the other hand, $S$ contains the kite $xr_0ys_0$, so $S'$ has to contain a $d_0$-wide vertical segment of this kite including the middle. Therefore, we have $\mu(S')\ge\diam(S)d(r_0,s_0)(1/4-(1/2-d_0/\diam(S))^2)$ and $\mu(S')/\mu(S)\ge(1/4-(1/2-d_0/\diam(S))^2)= d_0/\diam(S)-(d_0/\diam(S))^2\ge d_0/(2\diam(S))$ as claimed.

Let $c_1$ (resp., $c_2$) stand for the maximum distance from the line $xy$ of any point of $S'$ above (resp., below) of $xy$. We have $c:=c_1+c_2\ge d(r_0,s_0)=C(S,x,y)d(x,y)\ge C(S)\diam(S)$. The height of $r$ above the $xy$ line is at least $(1-d_0/(d(x,y)/2))c_1=(1-2d_0/\diam(S))c_1$. Similarly, the distance of $s$ from $xy$ is at least $(1-2d_0/\diam(S))c_2$, so $d(r,s)\ge(1-2d_0)/\diam(S))c$. The projection of $S'$ to the $xy$ line is an interval of length $d_0$ while the projection to a vertical line is an interval of length $c$, thus we have $d(p,q)\le\diam(S')\le d_0+c$. We can calculate $C(S',p,q)=d(r,s)/d(p,q)\ge(1-2d_0/\diam(S))c/(d_0+c)$. This bound is increasing in $c$ and decreasing in $d_0$, so using the estimates $c\ge C(S)\diam(S)$ and $d_0\le2\diam(S)\mu(S')/\mu(S)$ we obtain
$$C(S',p,q)\ge\left(1-4\frac{\mu(S')}{\mu(S)}\right)\frac{C(S)}{2\frac{\mu(S')}{\mu(S)}+C(S)}.$$
This suffices for part~2 of the lemma and finishes the proof of the lemma.
\end{proof}

The first part of the last lemma is enough to prove Theorem~\ref{fat}.

\begin{proof}[Proof of Theorem~\ref{fat}.]
We cut $S$ into $2^k$ part rsecursively halving it $k$ times as follows. At step $i$ we have cut $S$ into $2^i$ pairwise internally disjoint convex compact subregions of equal area. Then we cut each of these subregions into two parts by a straight line cut perpendicular to a segment connecting two diagonally opposite points of the subregions (that is, two points of the subregion, whose distance is the diameter of the subregion). By Lemma~\ref{sections}/1, the roundness of any subregion remains at least $\min(C(S),1/12)$. The $2^k$ subregions obtained after step $k$ satisfy the requirements of the theorem.
\end{proof}

Lemma~\ref{sections} estimated how cuts in a special direction change the roundness of a convex compact region in the plane. In the next lemma we work with a bisection to equal area parts in an \emph{arbitrary} direction.

\begin{lemma}\label{longcut}
Let $S$ be a compact convex subset of the plane of positive area and let $l$ be a line cutting it into two parts of equal area. Let $S'$ be one of the two parts and let $pq$ be the segment $l\cap S$.
\begin{enumerate}
\item $C(S')\ge C(S)/4$.
\item $d(p,q)\ge\frac{\mu(S)}{\diam(S)}$.
\item The middle half of the segment $pq$ is at a distance at least $\frac{C(S)d(p,q)}{4\sqrt{4+C(S)^2}}$ from the boundary of $S$.
\end{enumerate}
\end{lemma}

\begin{proof}
For part~1, we use only $\diam(S')\le\diam(S)$ and Lemma~\ref{compare}/2:
$$C(S')\ge D(S')=\frac{\mu(S')}{(\diam(S'))^2}\ge\frac{\mu(S)}{2(\diam(S))^2}=\frac{D(S)}2\ge\frac{C(S)}4.$$

Let us use $d_0=d(p,q)$. To see part~2 of the lemma, find the two tangent lines $l_1$ and $l_2$ of $S$ parallel to $l$ and also find tangent lines $l'$ and $l''$ of $S$ going through $p$ and $q$. These four lines bound a possibly degenerate trapezoid containing $S$, see Figure~\ref{fig:equal-area-supports}. Thus, the area of the part of this trapezoid on either side of $l$ is at least  $\mu(S)/2$. If $d_i$ is the distance between $l_i$ and $l$ and $\alpha$ is chosen appropriately (depending on the angles of the trapezoid), the two areas can be expressed as $d_0d_1-\alpha d_1^2$ and $d_0d_2+\alpha d_2^2$. We add up the lower bounds on these areas with coefficients to cancel the terms involving $\alpha$ and get $d_0\ge\frac{d_1^2+d_2^2}{2d_1d_2}\cdot\frac{\mu(S)}{d_1+d_2}\ge\frac{\mu(S)}{\diam(S)}$ as claimed.

For part~3, let $z$ be the midpoint of the segment $pq$ and let the segment $rs$ be the intersection of $S$ and the perpendicular bisector of the segment $pq$, see Figure~\ref{fig:middle-half-kite}. The kite $prqs$ is contained in $S$, so it is enough to bound the distance of the middle half of $pq$ from the four edges of this kite. By symmetry, it is enough to deal with the single edge $pr$. Clearly, the closest point to this edge in the middle half of $pq$ is the midpoint $z^*$ of $pz$. Let $d^*$ stand for the distance of $z^*$ from $pr$. From similar triangles we have
$$\frac{d^*}{d(p,z^*)}=\frac{d(r,z)}{d(p,r)}.$$
From the Pythagorean theorem, $d(p,r)=\sqrt{d(p,z)^2+d(z,r)^2}$. We also have $d(p,z)=d_0/2$ and $d(p,z^*)=d_0/4$.

\begin{figure}[!htbp]
\centering
\begin{minipage}[t]{0.49\textwidth}
\vspace{0pt}\centering
\includegraphics[width=\linewidth]{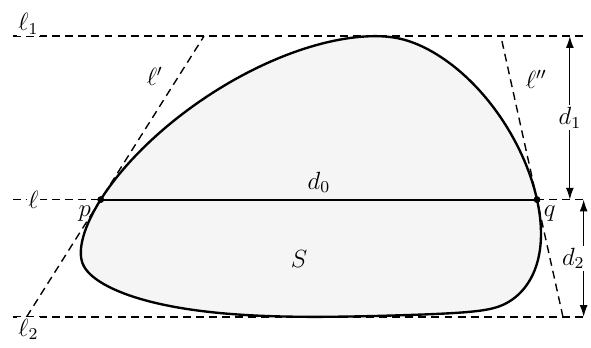}
\caption{}\label{fig:equal-area-supports}
\end{minipage}\hfill
\begin{minipage}[t]{0.49\textwidth}
\vspace{0pt}\centering
\includegraphics[width=\linewidth]{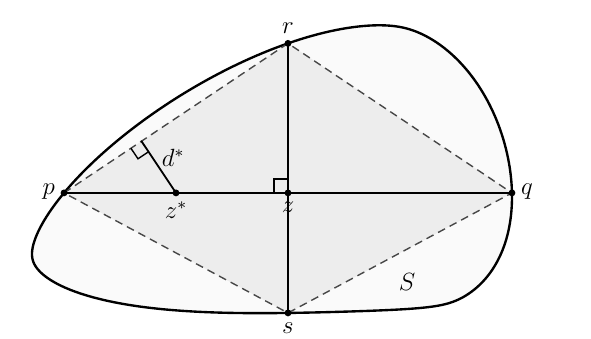}
\caption{}\label{fig:middle-half-kite}
\end{minipage}
\end{figure}

The part of $S$ on one side of $l$ contains the triangle $pqs$, while the part of $S$ on the other side is contained in the possibly degenerate quadrilateral bounded by parts of the lines $l$, $ps$, $qs$ and a tangent of $S$ going through $r$. The area of the triangle is thus at most the area of the quadrilateral. This implies through elementary arguments that $d(r,z)\ge d(r,s)/4$. We further have $d(r,s)=C(S,p,q)d_0\ge C(S)d_0$.

Putting all these estimates together we obtain
$$d^*=\frac{d_0d(r,z)}{4\sqrt{d_0^2/4+d(r,z)^2}}\ge\frac{d_0\cdot C(S)d_0/4}{4\sqrt{d_0^2/4+(C(S)d_0/4)^2}}=\frac{C(S)d_0}{4\sqrt{4+C(S)^2}}$$
as claimed.
\end{proof}

\section{The construction of a finite graph}\label{sec3}

Fix a compact convex region $S_*$ in the plane with roundness $C(S_*)>1/500$. $S_*$ will be a large circular disk in the applications. For any non-negative integer $k$, we will split $S_*$ into $40^k$ compact, convex, pairwise internally disjoint, equal area parts we call the \emph{level $k$ regions}. This process is very similar to the one we used to prove Theorem~\ref{fat}, but slightly more involved to ensure that the directions of the cuts we use are distributed close to uniformly.

The level $k$ regions are denoted by $S_u$ with $u\in I^k$, where $I$ stands for the set of the first $40$ non-negative integers. We use the word \emph{region} exclusively to denote one of these level $k$ regions for some $k$. We define the level $k$ regions by recursion starting with $S_*$ being the sole level $0$ region (here $*$ represents the empty sequence for consistency). If $S_u$ is a  level $k$ region, then we split it into $40$ level $k+1$ regions (the \emph{subregions of $S$}). We do this as follows. First we identify $x,y\in S_u$ with $d(x,y)=\diam(S_u)$ and split $S_u$ into 40 parts of equal area using 39 straight line cuts orthogonal to $xy$. From the 40 parts obtained this way we pick two consecutive ones such that one of them contains the midpoint of the segment $xy$. The other 38 parts obtained are subregions of $S_u$, but we further process the two chosen parts as follows: we split their union by a straight line $l$ of a given direction into two equal area parts. These two parts are called the \emph{central subregions} of $S_u$. They are subregions of $S_u$ and level $k+1$ regions. We denote the subregions of $S_u$ by $S_{uw}$, where $w$ runs through $I$. It remains to define the direction of the line $l$ separating the two central subregions of $S_u$. If $u=u_1u_2\ldots u_k$, then we turn the $x$ axis by the positive angle $\pi\sum_{i=1}^ku_i40^{i-k-1}$ and pick $l$ parallel to the line so obtained.

\begin{lemma}\label{kerek}
\begin{enumerate}
\item All regions $S$ satisfy $C(S)>1/500$;
\item the union $T$ of the two central subregions of a region satisfies $C(T)>1/125$.
\end{enumerate}
\end{lemma}

\begin{proof}
We prove the part~1 by induction on the level of the region and prove part~2 also during this induction. The base is the sole level $0$ region $S_*$ and we have assumed that it satisfies part~1.

Assume a region $S$ satisfies $C(S)>1/500$ and we prove the same inequality for its subregions and also that $C(T)>1/125$ for the union $T$ of its central subregions.

By Lemma~\ref{sections}/2 we have $C(T)\ge\min((4/5)C(S)/(1/10+C(S)),1/120)>1/125$ proving part~2 for $T$. For a central subregion $S'$ of $S$, Lemma~\ref{longcut}/1 implies $C(S')\ge C(T)/4>1/500$ as needed. Finally, for a non-central subregion $S'$ of $S$ Lemma~\ref{sections}/1 states $C(S')\ge\min(C(S),1/240)>1/500$ as needed.
\end{proof}

We build a discrete point set $P'$ inside $S_*$ as follows: If $S$ is a region of any level with $\mu(S)\ge10^{10}$, then we find the boundary between its two central subregions and place as many equidistant points at unit distances from each other on the middle half of this segment as will fit. Let $P'$ consist of all these points placed for any (large enough) region.

\begin{lemma}\label{P'}
\begin{enumerate}
\item The distance between any two distinct points in $P'$ is at least $1$.
\item Let $k$ be a non-negative integer, let $S$ be a region of area at least $10^{10}\cdot40^k$ and let $L$ be any line in the plane. Then $S$ contains a segment $pq$ of integer length $$d(p,q)\ge\frac{\sqrt{\mu(S)}}{150\cdot40^{k/2}},$$ whose direction makes an angle at most $\pi/40^k$ with the line $L$, such that $P'$ contains all the points of the segment $pq$ at integer distance from $p$, including $p$ and $q$.
\end{enumerate}
\end{lemma}

\begin{proof}
For part~1 consider the distinct points $p,p'\in P'$ and assume that they are on the boundary between the central subregions of the regions $S$ and $S'$, respectively. We assume without loss of generality that $S$ and $S'$ are level $k$ and level $k'$ regions, respectively with $k\le k'$. If $S=S'$, then the distance $d(p,p')$ is an integer. Otherwise, $p$ must be outside or on the boundary of $S'$. By the estimates in Lemma~\ref{longcut}/2,3 and the bound $C(T)>1/125$ in Lemma~\ref{kerek}/2 for the union $T$ of the two central subregions of $S'$, if $ab$ denotes the segment separating these two central subregions, then $p'$ lies in the middle half of $ab$, and

$$d(p,p')\ge \frac{C(T)d(a,b)}{4\sqrt{4+C(T)^2}}>\frac{d(a,b)}{1001},$$ $$d(a,b)\ge\frac{\mu(T)}{\diam(T)}=\sqrt{D(T)\mu(T)}\ge\sqrt{\frac{C(T)\mu(T)}2}\ge\sqrt{\frac{\mu(S')}{5000}}>1400.$$
For the subsequent (in)equalities, we used $D(T)=\mu(T)/(\diam(T))^2\ge C(T)/2$ from Lemma~\ref{compare}/2, $C(T)\ge1/125$ from Lemma~\ref{kerek}/2, $\mu(T)=\mu(S')/20$, and $\mu(S')\ge10^{10}$.

For part~2, let $S=S_u$ and consider the regions $S_{uv}$ with $v$ running through the sequences in $I^k$. These smaller regions are inside $S$ and have area $\mu(S)/40^k\ge10^{10}$, so we have points in $P'$ on the boundary between their central subregions. Note that the $40^k$ directions of these boundaries are evenly distributed in the full circle, so one of them is within angle $\pi/40^k$ of the line $L$. Let this boundary be the segment $p_0q_0$ separating the central subregions $S'$ and $S''$ of $S_{uv}$. With $T=S'\cup S''$, Lemma~\ref{longcut}/2 yields $d(p_0,q_0)\ge\frac{\mu(T)}{\diam(T)}$ and the same calculation we saw in the proof of part~1 above yields
$$d(p_0,q_0)\ge\frac{\mu(T)}{\diam(T)}\ge\sqrt{\frac{\mu(T)}{250}}=\sqrt{\frac{\mu(S)}{5000\cdot40^k}}>1400.$$

Within the middle half of the segment $p_0q_0$ we find the segment $pq$ satisfying all the requirements of part~2 of the lemma including
$$d(p,q)=\left\lfloor\frac{d(p_0,q_0)}2\right\rfloor\ge\frac{\sqrt{\mu(S)}}{150\cdot 40^{k/2}}$$.
\end{proof}

Let us extend the set $P'$ to a maximal set $P$ of points in $S_*$ with minimum distance at least $1/3$. That is, $P'\subseteq P\subset S_*$, the distance between any two distinct points in $P$ is at least $1/3$ and for any point in $S_*,$ there is a point in $P$ within distance $1/3$. One can find such a set $P$ by greedily adding points to $P'$ violating the last condition one by one until all conditions are satisfied.

We define the \emph{subdivison graph} $H$ on the vertex set $P$ by connecting $p,q\in P$ if and only if $d(p,q)\le1$. Let $d_H(x,y)$ denote the graph distance between $x,y\in P$.

\section{Distance bounds on the finite subdivision graph}\label{sec4}

This calculation is motivated by a similar calculation for the \emph{pinwheel tiling}, see Cai et al., \cite{CaC26}. Our bounds are better than the ones for the pinwheel tiling because the direction of the separating lines are distributed evenly.

We say that $H$ satisfies the \emph{$(m,\varepsilon)$-bound} if for any vertices $x$ and $y$ of $H$ with $d(x,y)\ge m$ we have $d_H(x,y)\le(1+\varepsilon)d(x,y)$. The next lemma establishes that $H$ satisfies some $(m,\varepsilon)$-bound, while the subsequent one improves the error parameter $\varepsilon$ in the bound for the price of increasing the threshold $m$.

\begin{lemma}\label{base2}
$H$ satisfies the $(0,2)$-bound.
\end{lemma}

\begin{proof}
The bound is satisfied for $x=y$ with $d(x,y)=d_H(x,y)=0$ and for $0<d(x,y)\le1$ with $d(x,y)\ge1/3$ and $d_H(x,y)=1$. Assume $d(x,y)>1$, set $j=\lceil3d(x,y)\rceil-3$ and place the point $p_i$ on the segment $xy$ at distance $(i+1)/3$ from $x$ for $1\le i\le j$. By the convexity of $S_*$, these points are in $D$, so by the construction of $P$ we find points $x_i\in P$ within distance $1/3$ of $p_i$. By the triangle inequality, we have $d(x,x_1)\le1$, $d(x_i,x_{i+1})\le1$ for $1\le i<j$, and $d(x_j,y)\le 1$, so there is a path connecting $x$ and $y$ through the vertices $x_i$ and so $d_H(x,y)\le j+1<3d(x,y)$ as needed.
\end{proof}

\begin{lemma}\label{step2}
If $H$ satisfies the $(m,\varepsilon)$-bound ($\varepsilon\le2$), then $H$ also satisfies the $(m',\varepsilon')$-bound with $m'=\max(3m,10^{16}\varepsilon^{-3/2})$ and $\varepsilon'=\varepsilon-\varepsilon^{5/2}/10^{14}$.
\end{lemma}

\begin{proof}
Let $k$ be a positive integer such that $60/\varepsilon\le40^{2k}<96\cdot10^3/\varepsilon$ and set $a=\frac{\varepsilon^2}{10^{16}\cdot40^k}$.

Let $x$ and $y$ be vertices of $H$ with $d_0:=d(x,y)\ge m'$. Our goal is to prove $d_H(x,y)\le(1+\varepsilon')d_0$. Let $z$ be the midpoint of the segment $xy$. Note that $x$ and $y$ are in the convex region $S_*$, so $z\in S_*$ and thus covered by a region of every level. Let $S$ be the smallest region containing $z$ with $\mu(S)\ge ad_0^2$. Note that $\mu(S_*)=D(S_*)\diam(S_*)^2\ge d_0^2/1000$, so $S_*$ is large enough and therefore the minimal large-enough region $S$ containing $z$ exists. The minimality implies that $\mu(S)<40ad_0^2$. Using Lemma~\ref{P'}/2, we find $p,q\in S\cap P'$ with $d_H(p,q)=d(p,q)\ge\sqrt{\mu(S)}/(150\cdot40^{k/2})$ and the angle between the lines $xy$ and $pq$ at most $\pi/40^k$. For the lemma to be applicable here, we need $\mu(S)\ge10^{10}\cdot40^k$ which follows from $\mu(S)\ge ad_0^2$, $d_0\ge m'\ge10^{16}\varepsilon^{-3/2}$, and $40^{2k}<96\cdot10^3/\varepsilon$. See Figure~\ref{fig:finite-step}.

\begin{figure}[!htbp]
\centering
\includegraphics[width=0.62\textwidth]{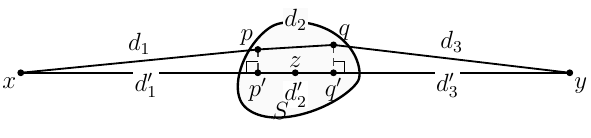}
\caption{}\label{fig:finite-step}
\end{figure}

Let $d_1=d(x,p)$, $d_2=d(p,q)$ and $d_3=d(q,y)$.
Notice that $\diam(S)=\sqrt{\mu(S)/D(S)}\le\sqrt{1000\mu(S)}\le200\sqrt ad_0\le d_0/6$ by Lemmas~\ref{kerek}/1 and \ref{compare}/2 and therefore the segment $pq$ is relatively close to the midpoint $z$ of the segment $xy$. Recall that $pq$ is close to parallel to $xy$. Switch the roles of $p$ and $q$, if necessary to ensure that $p$ is closer to $x$ than $q$. We have $d_1\ge d(x,z)-d(z,p)\ge d_0/2-\diam(S)\ge d_0/3\ge m$. Similarly, $d_3\ge d_0/3\ge m$ and so we can apply the $(m,\varepsilon)$-bound on both pairs of vertices. We obtain
$$d_H(x,y)\le d_H(x,p)+d_H(p,q)+d_H(q,y)\le(1+\varepsilon)d_1+d_2+(1+\varepsilon)d_3.$$

Let $p'$ and $q'$ be the orthogonal projections of $p$ and $q$ to the line $xy$ and let $d_1'=d(x,p')$, $d_2'=d(p',q')$, and $d_3'=d(q',y)$. Clearly,
$$d_0=d_1'+d_2'+d_3'.$$
Here $d_2'=d_2\cos\theta$, where $\theta<\pi/40^k$ is the angle between $xy$ and $pq$. We have $d(p',z)\le d(p,z)\le d_0/6$, so $d'_1\ge d_0/3$. Using $d(p,p')\le d(p,z)\le\diam(S)\le\sqrt{1000\mu(S)}$ and the Pythagorean theorem, we obtain $$d_1=\sqrt{d_1'^2+d(p,p')^2}\le\sqrt{d_1'^2+1000\mu(S)}<d_1'+\frac{1000\mu(S)}{2d_1'}\le d'_1+1500\frac{\mu(S)}{d_0}.$$
By symmetry, we also have
$$d_3<d_3'+1500\frac{\mu(S)}{d_0}.$$
Combining all these estimates, we obtain

\begin{eqnarray*}
d_H(x,y)&\le&(1+\varepsilon)(d_1+d_3)+d_2\\
&\le&(1+\varepsilon)\left(d_0-d_2'+3000\frac{\mu(S))}{d_0}\right)+d_2\\
&=&(1+\varepsilon)d_0-(\varepsilon-(1+\varepsilon)(1-\cos\theta))d_2+3000(1+\varepsilon)\frac{\mu(S)}{d_0}\\
\end{eqnarray*}

In the last formula $(1+\varepsilon)d_0$ is the estimate the $(m,\varepsilon)$-bound gives for $d_H(x,y)$. We need to estimate the other two terms. We use $(1+\varepsilon))\le3$, $1-\cos\theta\le\theta^2/2\le\pi^2/(2\cdot40^{2k})$, and $40^{2k}\ge60/\varepsilon$ to conclude $(1+\varepsilon)(1-\cos\theta)\le \varepsilon/4$. This allows us to estimate
$$(\varepsilon-(1+\varepsilon)(1-\cos\theta))d_2\ge\frac{3\varepsilon}4 d_2\ge\frac{\varepsilon\sqrt{\mu(S)}}{200\cdot40^{k/2}}\ge\frac{\varepsilon\sqrt ad_0}{200\cdot40^{k/2}}=\frac{\varepsilon^2}{2\cdot10^{10}\cdot40^k}d_0.$$
We use $\mu(S)<40ad_0^2$ to estimate the last term:
$$3000(1+\varepsilon)\frac{\mu(S)}{d_0}\le36\cdot10^4ad_0=\frac{36\varepsilon^2}{10^{12}\cdot40^k}d_0.$$
We can now finish the proof of the lemma by returning to the estimate for $d_H(x,y)$:
\begin{eqnarray*}
d_H(x,y)&\le&(1+\varepsilon)d_0-(\varepsilon-(1+\varepsilon)(1-\cos\theta))d_2+3000(1+\varepsilon)\frac{\mu(S)}{d_0}\\
&\le&(1+\varepsilon)d_0-\frac{\varepsilon^2}{2\cdot10^{10}\cdot40^k}d_0+\frac{36\varepsilon^2}{10^{12}\cdot40^k}d_0\\
&=&(1+\varepsilon)d_0-\left(\frac{1}{2\cdot10^{10}}-\frac{36}{10^{12}}\right)\frac{\varepsilon^2}{40^k}d_0\\
&\le&(1+\varepsilon)d_0-\frac{\varepsilon^{5/2}}{10^{14}}d_0=(1+\varepsilon')d_0
\end{eqnarray*}
\end{proof}

Let us define $(m_i,\varepsilon_i)$ recursively starting at $(m_0,\varepsilon_0)=(0,2)$ and letting $(m_{i+1},\varepsilon_{i+1})=(m',\varepsilon')$ as obtained in Lemma~\ref{step2} from $(m,\varepsilon)=(m_i,\varepsilon_i)$. The graph $H$ satisfies the $(m_i,\varepsilon_i)$-bound by Lemmas~\ref{base2} and \ref{step2}. The values $m_i$ and $\varepsilon_i$ are determined by the formulae in Lemma~\ref{step2} and are independent of the size of our starting region $S_*$. The $(m_i,\varepsilon_i)$-bound holds for $H$ for any $i$ but as $m_i$ grows beyond $\diam(S_*)$ it becomes meaningless. Nevertheless, it is useful to calculate the growth rate of $m_i$ and $\varepsilon_i$. The recursion $\varepsilon_{i+1}=\varepsilon_i-\Theta(\varepsilon_i^{5/2})$ ensures $\varepsilon_i=\Theta(i^{-2/3})$. The recursion $m_{i+1}=\max(3m_i,O(\varepsilon_i^{-3/2}))$ and $m_1>0$ imply $m_i=\Omega(3^i)$, which, in turn, implies that $m_{i+1}=3m_i$ for all but finitely many indices $i$ and so $m_i=\Theta(3^i)$.


\section{The infinite graph--Proof of Theorem~\ref{generalizedpinwheel}}\label{sec5}

Let us denote by $B_r(p)$ the ball of radius $r$ around the point $p$ in the plane. We use $O$ to denote the origin.

Consider any sequence of positive numbers $r_j$ ($j\ge1$) that tends to infinity and let us start the construction described in the previous section with $S_*=B_{r_j}(O)$. Let $H_j$ be the subdivision graph so obtained and $P_j$ its vertex set. Let us number the vertices in $H_j$ in order of their distance from the origin (breaking ties arbitrarily) to obtain $P_j=\{p_j^1,p_j^2,\dots,p_j^{|P_j|}\}$.

As the balls $B_{1/3}(p_j^i)$ for fixed $j$ and $1\le i\le|P_j|$ cover $B_{r_j}(O)$, $|P_j|$ must tend to infinity with $j$. For a fixed index $i$, $p_j^i$ must therefore be defined for all large enough $j$. The $i_0$ balls $B_{1/3}(p_j^i)$ for fixed $j$ and $i\le i_0$ must cover the interior of $B_r(O)$ for $r=d(O,p_j^{i_0})-1/3$ because we indexed the points in increasing order of their distance from the origin, and therefore $d(O,p_j^{i_0})$ must be bounded for each fixed $i_0$. The compactness of these bounded balls implies that for a subsequence $j_1<j_2<\cdots$ the sequences $p_{j_k}^i$ must be convergent for all $i$. Let $p^i$ stand for the limit of the sequence $P_{j_k}^i$ and let $P=\{p^i\mid i\ge1\}$. Finally, let $H$ be the infinite graph on the vertex set $P$ with two vertices adjacent if and only if their distance is at most $1$.

\begin{proof}[Proof of Theorem~\ref{generalizedpinwheel}]
We claim that the set $P$ and graph $H$ constructed above satisfy the requirements listed in the theorem.

Part~1 follows from $d(p^{i_1}p^{i_2})=\lim_kd(p^{i_1}_{j_k},p^{i_2}_{j_k})$ and $d(p_{j_k}^{i_1}p_{j_k}^{i_2})\ge1/3$ whenever $i_1\ne i_2$.

Let $p$ be any point in the plane. As $B_{r_j}(O)$ contains $p$ for all large enough $j$ there must be a point $p^{i_j}_j$ within distance $1/3$ of $p$ for all these indices $j$. The balls $B_{1/6}(p_j^i)$ are internally disjoint for any fixed $j$ and for $i\le i_j$ they are inside $B_r(O)$ for $r=d(O,p)+1/2$. Therefore $i_j$ must be bounded and thus there exists a value $i$ such that $i_{j_k}=i$ for infinitely many $k$. But then $d(p,p^i)=\lim_kd(p,p^i_{j_k})\le1/3$ proving part~2.

The inequality $d_H(p,q)\ge d(p,q)$ in part~3 is a consequence of no edge of $H$ connecting vertices of distance greater than $1$.

It remains to prove the upper bound on the graph distance in part~3. Let $p=p^{i_1}$ and $q=p^{i_2}$ be two distinct vertices of $H$. Recall the sequences $m_i$ and $\varepsilon_i$ defined at the end of the last section and let $i_0$ be the largest index with $m=m_{i_0}<d(p,q)$ and let $\varepsilon=\varepsilon_{i_0}$. For large enough $k$, the vertices $p_{j_k}^{i_1}$ and $p_{j_k}^{i_2}$ exist and their distance is larger than $m$ as $m<d(p,q)=\lim_kd(p^{i_1}_{j_k},p^{i_2}_{j_k})$. As the graph $H_{j_k}$ satisfies the $(m,\varepsilon)$-bound, we have $d_{H_{j_k}}(p^{i_1}_{j_k},p^{i_2}_{j_k})\le(1+\varepsilon)d(p^{i_1}_{j_k},p^{i_2}_{j_k})$. In the limit, this gives $d_H(p,q)\le(1+\varepsilon)d(p,q)$ because $d(p,q)=\lim_kd(p^{i_1}_{j_k},p^{i_2}_{j_k})$ and $d_H(p,q)\le\liminf_kd_{H_{j_k}}(p^{i_1}_{j_k},p^{i_2}_{j_k})$. To see the latter inequality, it is enough to notice that if the graph distance is some fixed value $c$ for infinitely many graphs $H_{j_k}$, then the length $c$ path connecting $p^{i_1}_{j_k}$ and $p^{i_2}_{j_k}$ must go through vertices in a bounded distance from the origin, therefore vertices $p^l_{j_k}$ for bounded indices $l$, and thus it must be the path $p^{l_0}_{j_k}p^{l_1}_{j_k}\cdots p^{l_c}_{j_k}$ for the same indices $i_1=l_0,l_1,\dots,l_c=i_2$ for infinitely many values of $k$. But then $p^{l_0}p^{l_1}\cdots p^{l_c}$ is a length $c$ path connecting $p$ and $q$ in $H$.

We saw that $m_i=\Theta(3^i)$ and $\varepsilon_i=\Theta(i^{-2/3})$. This shows that $i_0=\Theta(\log(d(p,q)))$ and, thus, $$\varepsilon_{i_0}=\Theta(\log^{-2/3}(d(p,q)))$$ giving us the bound claimed.
\end{proof}

\section{Proof of Theorem~\ref{best}}\label{sec6}
This section follows the proof of the main result in the paper \cite{CaC26} with minor, but important, adjustments. Note that the original result used weights between $\sqrt2/2$ and $2$ on the edges of the square grid to achieve graph distances close to the Euclidean distances. We ``discretize'' this by using just two possible weights, which can be subsequently turned into an unweighted graph. The original result included a somewhat higher, but still polynomial error term. With a little effort we obtain a smaller error term, but note that the exponent can be further decreased, for example using recursive bounds for some of the connecting intervals in the proof below. We will often refer to the proof in \cite{CaC26}.

\begin{theorem}\label{ab}
Let $0\le a\le\sqrt2/2$ and $b>1$ be reals. There is an assignment of lengths to the edges of the standard unit square grid graph in the plane, where the length of each edge is either $a$ or $b$ and for the resulting weighted graph $H$ and vertices $u$ and $v$ we have
$$d(u,v)-O(1)\le d_H(u,v)=d(u,v)+O(d(u,v)^c),$$
where $c=(11+4\sqrt6)/25$. The hidden constants in the asymptotic notation depend on the choice of $a$ and $b$.
\end{theorem}

\begin{proof}
Let us denote the unweighted square grid graph by $G$. Following the paper \cite{CaC26}, we identify \emph{highways} corresponding to segments in the plane. Let $I$ be a segment in the plane, and let the corresponding highway, Highway$(I)$ be defined as the subgraph of $G$ induced by the vertices $v=(x,y)$ such that the square $[x-1/2,x+1/2)\times[y-1/2,y+1/2)$ intersects $I$. We can similarly assign Highway$(\ell)$ to a full line $\ell$. As observed in \cite{CaC26},  Highway$(I)$ is a path. By Lemma~3.1 of \cite{CaC26}, $|\tau d_G(u,v)-d(u,v)|\le1$ for any two vertices $u,v\in$ Highway$(I)$, where $\tau=\sqrt{s^2+1}/(|s|+1)$ and $s$ is the slope of $I$, that is, the line of $I$ is given by $y=sx+c_0$ (for vertical $I$, $\tau=1$). Note that $\sqrt2/2\le \tau\le1$.

The observation above tells us that assigning length $\tau$ to all edges in Highway$(I)$, the resulting weighted graph-distance approximates the Euclidean distance. We use $a\le \tau<b$ to use a different weighting we call the \emph{$(a,b)$-weighting} of Highway$(I)$. We obtain the weighted path $P$ by assigning length $a$ or $b$ to the edges of Highway$(I),$ recursively, starting at one end vertex $u_0$. After assigning lengths to the edges in the path from $u_0$ to some vertex $v$ in Highway$(I)$, we look at the following edge $vw$ in the path (if exists) and assign to it the length $a$ if $d_P(u_0,v)\ge \tau d_G(u_0,v)$ and the length $b$, otherwise. Clearly, we have $$d_P(u_0,w)-\tau d_G(u_0,w)=d_P(u_0,v)-\tau d_G(u_0,v)+a-\tau$$ in the former case, and $$d_P(u_0,w)-\tau d_G(u_0,w)=d_P(u_0,v)-\tau d_G(u,v)+b-\tau$$ in the latter. Thus, the inequalities $$-b\le a-\tau\le d_P(u_0,v)-\tau d_G(u_0,v)\le b-\tau\le b$$ will be maintained throughout the path. For any two vertices $u,v\in$ Highway$(I)$, we have
\begin{eqnarray*}
|d_P(u,v)-d(u,v)|&=&\big||d_P(u_0,u)-d_P(u_0,v)|-d(u,v)\big|\\
&\le&\big||\tau d_G(u_0,u)-\tau d_G(u_0,v)|-d(u,v)\big|+2b\\
&=&|\tau d_G(u,v)-d(u,v)|+2b\\
&\le&2b+1.
\end{eqnarray*}

With $\alpha=\sqrt6/6,$ we set the parameters $k_i$ and $m_i$ recursively starting at $m_1$ large enough to ensure $(b-1)(m_1-2)>2b+1$ and setting $k_i=\lceil m_i^\alpha\rceil$, $m_{i+1}=k_i^2m_i$.

The \emph{$d$-neighborhood} of a planar set $S$ means the set of points that are within distance $d$ from $S$ (the Minkowski sum of $S$ and the closed radius $d$ ball around the origin). Following \cite{CaC26} we define the lines
$$\ell_{i,j,t}=\{(x,y)\mid y\cos\frac{j\pi}{k_i}=x\sin\frac{j\pi}{k_i}+(2t+1)k_i^3m_i\},$$
for $i=1,2,\dots$, $j=1,2,\dots k_i$ and $t\in\ZZ$. Let Lines$(i)=\{\ell_{i,j,t}\}$. This set contains lines in $k_i$ equidistant directions in the plane, with lines in regular distances of $2k_i^3m_i$ in each of these directions. Note that all lines in Lines$(i)$ are at a minimum distance of $k_i^3m_i$ from the origin. We define Segments$(i)$ as a collection of intervals obtained from Lines$(i)$ by the following process. For each $\ell\in$ Lines$(i)$ remove from $\ell$ the $m_i$-neighborhood of each line $\ell'\in$ Lines$(i)$, $\ell'\ne\ell$ and also the $m_i$-neighborhood of each segment in Segments$(j)$ for $j>i$. Whenever the $m_i$-neighborhood of a segment in Segments$(j)$ intersects $\ell$ in a nonempty segment of length shorter than $m_i$, remove a segment of length $m_i$ including the intersection. Whatever segments remain in $\ell$ for any line $\ell\in$ Lines$(i)$, constitute the set Segments$(i)$. This seems to be a badly formed recursion as Segments$(i)$ is defined using the sets Segments$(j)$ for all $j>i$. This problem can be fixed by restricting our attention to indices $i\le M$ in this hierarchy, then applying a compactness argument as $M$ goes to infinity, as done in \cite{CaC26}. Alternatively (and we choose this approach here), we notice that for $j>i$, any segment in Segments$(j)$ affects directly only its $2m_i$-neighborhood when defining Segments$(i)$, so even indirectly through intermediate sets Segments$(j')$ with $i<j'<j$, it only affects points in its $(2\sum_{j'=i}^{j-1}m_{j'})$-neighborhood. As points in Lines$(j)$ are at a minimum distance $k_j^3m_j$ from the origin, Segments$(j)$ has no effect on whether a point $p$ is removed from a line $\ell\in$ Lines$(i)$ in the process if $$d(p,O)<k_j^3m_j-2\sum_{j'=i}^{j-1}m_{j'}.$$ Because of the fast growth of $k_j$ and $m_j$, this will be true for any point $p$ and any large enough index $j$, so whether $p$ is removed from $\ell$ is well defined for any $p\in\ell\in$ Lines$(i)$. This makes Segments$(i)$ well defined.

In a direct analog of Lemma~3.2 in \cite{CaC26}, we observe that the distance between any two distinct segments in $\cup_{j\ge i}$Segments$(j)$ is at least $m_i$.

To obtain the weighted graph $H$, we assign lengths to the edges of the square grid $G$ as follows. For Highway$(I)$ with $I\in$ Segments$(i)$, $i\ge1$, we use its $(a,b)$-weighting, for all edges outside these highways we assign the length $b$. Note that the pairwise distance of all the segments in $\cup_{i\ge1}$Segments$(i)$ is at least $m_1>2$, so the highways are pairwise disjoint and thus there is no conflict in this definition.

Consider any vertices $p$ and $q$ in $H$. The distance $d_H(p,q)$ is realized by a path in $H$ and this path can be written as $B_0A_1B_1A_2\cdots A_kB_k$, where $A_j$ are maximal subpaths connecting vertices of the same highway without using edges of any other highway, while $B_j$ do not use highway edges and (except possibly for $B_0$ and $B_k$, which may be even empty) they connect vertices of distinct highways. Let $A_j$ connect $s_j$ with $t_j$ belonging to Highway$(I_j)$ for some $I_j\in$ Segments$(i_j)$. Clearly, the length of $A_j$ is at least the length of the path in Highway$(I_j)$ connecting $s_j$ and $t_j$, so at least $d(s_j,t_j)-(2b+1)$. We have $d(t_j,s_{j+1})\ge m_1-2$ as $t_j$ and $s_{j+1}$ are on separate highways and the corresponding intervals are at distance at least $m_1$. The $H$-length of $B_i$ is at least $b$ times the distance between its endpoints, which is at least $2b+1$ more than this distance if $1 \le i<k$. Here we used that the choice of $k_1$ ensures that $(b-1)(m_1-2)>2b+1$. Using this, we get
\begin{eqnarray*}
d_H(p,q)&\ge&d(p,s_1)+(d(s_1,t_1)-(2b+1))+(d(t_1,s_2)+(2b+1))+\cdots\\&&+(d(s_k,t_k)-(2b+1))+d(t_k,q)\\
&=&d(p,s_1)+d(s_1,t_1)+d(t_1,s_2)+\cdots+d(s_k,t_k)+d(t_k,q)-(2b+1)\\
&\ge&d(p,q)-(2b+1).
\end{eqnarray*}
This proves the lower bound in the theorem.
\smallskip

To see the upper bound, consider two distinct vertices $p$ and $q$ in $H$ and we construct a short path connecting them and passing close to two lines in Lines$(i)$. We will optimize for the parameter $i\ge1$ later.

Find line $\ell\in$ Lines$(i)$ such that the direction of $\ell$ is obtained from the line $pq$ by a rotation of a positive angle at most $\pi/k_i$ and its distance from $p$ is at most $k_i^3m_i$. Similarly, let $\ell'\in$ Lines$(i)$ have direction obtained from the line $pq$ by a negative rotation of an angle at most $\pi/k_i$ and let the distance of $q$ from $\ell'$ be at most $k_i^3m_i$. The definition of the set Lines$(i)$ ensures that such lines $\ell$ and $\ell'$ exist. Let $p'$ be the closest vertex in Highway$(\ell)$ to $p$, let $q'$ be the closest vertex in Highway$(\ell')$ to $q$. Note that $\ell$ and $\ell'$ intersect. We choose $r(x,y)$ to be the grid point with the square $[x-1/2,x+1/2)\times[y-1/2,y+1/2)$ containing this intersection point and note that $r$ is on the highway of both lines.

We route the path from $p$ to $q$ through $p'$, $r$ and $q'$. We will use that $d_H(u,v)\le bd_G(u,v)\le \sqrt2bd(u,v)=O(d(u,v))$ holds for any vertices $u$ and $v$ of $H$. Clearly, $d_H(p,p')=O(d(p,p'))=O(k_i^3m_i)$ and similarly, $d_H(q',q)=O(k_i^3m_i)$. We also have
$$d(p',r)+d(r,q')\le d(p,q)/\cos(\pi/k_i)+O(k_i^3m_i)=d(p,q)+O(d(p,q)/k_i^2+k_i^3m_i)$$. It remains to estimate $d_H(p'r)$ and $d_H(r,q')$.

To connect $p'$ and $r$, we try to follow $\ell$ and compare the $H$-distance traveled to how the closest point on $\ell$ progresses in Euclidean distance that we will simply call \emph{the progress}. Note that the total progress of the path connecting $p'$ and $r$ is $d(p',r)+O(1)$. The parts of $\ell$ \emph{not} deleted will result in an interval $I\subset\ell$, $I\in$ Segments$(i)$. Here we follow Highway$(I)$ and the $H$-length of this path is within an additive $O(1)$ from the progress. For segments of $\ell$ in the $m_i$-neighborhood of some other line in Lines$(i)$ we also follow Highway$(\ell)$. These deleted intervals have length $\Omega(m_i)$ and $O(m_ik_i)$ (because the angle between the lines is at least $\pi/k_i$). The deleted intervals from lines of a fixed direction are equal length, are equidistant and cover a ratio of exactly $1/k_i^3$ fraction of the line $\ell$, so they contribute $O(m_ik_i+d(p',r)/k_i^3)$ total distance. Summing for all the $k_i-1$ directions of lines in Lines$(i)$ intersecting $\ell$, we get a total Euclidean length of $O(m_ik_i^2+d(p',r)/k_i^2)$. This also means that the same bound holds for the $H$-length of the corresponding parts of the connecting path. Finally, for the parts of $\ell$ deleted because they are within distance $m_i$ of an interval $I\in$ Segments$(j)$ with $j>i$ (or the $m_i$-length interval containing a shorter intersection) we follow Highway$(I)$ with two $O(m_i)$ length connecting paths. The $H$-length of these intervals of our connecting path is therefore at most $O(m_i)$ plus the progress. As all of the segments in $\cup_{j>i}$Segments$(j)$ are separated by distance at least $m_{i+1}$, we have $O(1+d(p'r)/m_{i+1})$ such deleted intervals. Summarizing and using $m_{i+1}/m_i=k_i^2$:
\begin{eqnarray*}
d_H(p',r)&\le&d(p',r)+O(m_ik_i^2+d(p',r)/k_i^2)+O(m_i+d(p',r)m_i/m_{i+1})\\
&=&d(p',r)+O(d(p',r)/k_i^2+m_ik_i^2).
\end{eqnarray*}
We also have
$$d_H(r,q')=d(r,q')+O(d(r,q')/k_i^2+m_ik_i^2)$$
by symmetry. Putting all these estimates together we get:
\begin{eqnarray*}
d_H(p,q)&\le&d_H(p,p')+d_H(p',r)+d_H(r,q')+d_H(q',q)\\
&=&O(k_i^3m_i)+(d(p',r)+d(r,q'))(1+O(1/k_i^2))\\
&=&d(p,q)+O(k_i^3m_i+d(p,q)/k_i^2).
\end{eqnarray*}
This proves the upper bound in the theorem if $k_i^3m_i+d(p,q)/k_i^2=O(d(p,q)^c)$. There is nothing to prove if $d(p,q)^c<m_1k_1^3=O(1)$. Otherwise, we choose $i$ to be maximal with $k_i^3m_i\le d(p,q)^c$. We have
\begin{eqnarray*}
d(p,q)^c&<&k_{i+1}^3m_{i+1}\\
&=&O(m_{i+1}^{3\alpha+1})\\
&=&O(k_i^2m_i)^{3\alpha+1})\\
&=&O(m_i^{(2\alpha+1)(3\alpha+1)})
\end{eqnarray*}
As a consequence, $d(u,v)/k_i^2=O(d(u,v)/m_i^{2\alpha})=O(d(u,v)^{1-2\alpha c/((2\alpha+1)(3\alpha+1))})$. Here $\alpha$ and $c$ were chosen such that the complicated exponent in the last bound is exactly $c$. This finishes the proof of the theorem.
\end{proof}

The simplest way to apply Theorem~\ref{ab} to obtain an \emph{unweighted} graph in which the graph distance approximates the Euclidean distance is to set $a=0$. After contracting the length 0 edges and shrinking the vertex set by a factor of $b>1$, we get an unweighted planar graph whose graph distance approximates the Euclidean distance well. Alternatively, we can set $a=2/3$, $b=4/3$ in Theorem~\ref{ab}. Blowing up the vertex set of the weighted grid obtained by a factor of $3/2$ and similarly increasing the edge weights, the resulting grid approximates Euclidean distance with weights $1$ and $2$. Subdividing the weight $2$ edges by inserting a vertex at its middle, one obtains a unweighted graph with the same asymptotic properties.

\begin{figure}[!htbp]
\centering
\includegraphics[width=0.62\textwidth]{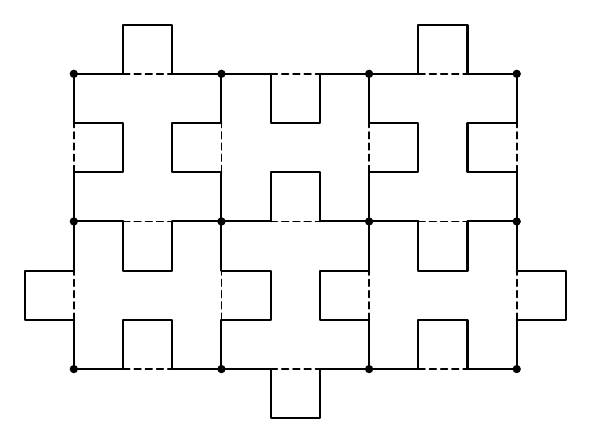}
\caption{}\label{fig:grid-gadget}
\end{figure}

\begin{proof}[Proof of Theorem~\ref{best}] For this proof we apply Theorem~\ref{ab} in slightly different way. We start with the unit grid graph with edge lengths $2/3$ and $10/9$ in which the graph distance approximates the Euclidean distance as stated in that theorem. We blow up the vertex set and the edge lengths by a factor of $9/2$ to obtain the weighted graph $H_0$, where all edge weights are $3$ or $5$. Now we consider the vertex set $V$ of a square grid of side length $3/2$ that contains all vertices of $H_0$. With a carefully selected subset of the grid edges, we form the graph $H$. Specifically, we include all solid edges depicted in Figure~\ref{fig:grid-gadget}. These edges provide a length 5 path between any pair of neighboring vertices in $H_0$ and also ensure that any vertex of $H$ is within distance 2 of a vertex of $H_0$. Then, for any pair of vertices in $H_0$ connected by a length 3 edge, we add a shortcut edge (depicted in Figure~\ref{fig:grid-gadget} as a dashed edge) to ensure that this pair of vertices is connected by length 3 path in $H$. The (unweighted) distance in the resulting graph $H$ between any two vertices in $H_0$ will agree with their weighted distance in $H_0$. This proves Theorem~\ref{best}.
\end{proof}

\smallskip

\textbf{Acknowledgement.} We used ChatGPT exclusively for beautifying the figures. We take full responsibility for our mistakes (if there are any).


\begin{thebibliography}{40}
\bibitem{AhBD20}
R.~Ahmed, G.~Bodwin, F.~Darabi Sahneh, K.~Hamm, M.~J.~Jebelli, S.~Kobourov, and R.~Spence. Graph spanners: A tutorial review. \emph{Computer Science Review} \textbf{37} (2020), 100253.

\bibitem{Be13}
I.~Benjamini: Euclidean vs. graph metric. In: \emph{Erd\H os Centennial,} Springer, Berlin, Heidelberg, 2013, 35--57.

\bibitem{Be26}
I.~Benjamini: Euclidean vs. graph metric: The fixed-source problem. arXiv:2606.13271 (2026).

\bibitem{BeSA17}
M.~Bern, J.~R.~Shewchuk, and N.~Amenta: Triangulations and mesh generation. In: \emph{Handbook of Discrete and Computational Geometry
(J.~E.~Goodman et al., eds), 3rd edition}, CRC Press, Boca Raton, FL, 2017, 763--785.

\bibitem{BrS94}
S.~C.~Brenner and L.~R.~Scott: \emph{The Mathematical Theory of Finite Element Methods}, Springer, New York, 1994.

\bibitem{BuI15}
D.~Burago and S.~Ivanov: Uniform approximation of metrics by graphs. \emph{Proceedings of the American Mathematical Society} \textbf{143} (3) (2015), 1241--1256.

\bibitem{CaC26}
Z.~Cai, K.~Chen, Sh.~Du, A.~Filtser, S.~Pettie, and D.~Skora:
The squishy grid problem. In: \emph{42nd International Symposium on Computational Geometry (SoCG 2026), LIPIcs} \textbf{367}, 2026, 27:1--27:16.
doi:10.4230/LIPIcs.SoCG.2026.27.
  
\bibitem{Ch86}
P.~Chew: There is a planar graph almost as good as the complete graph. In: \emph{ Proc. 2nd Ann. Symposium on Computational Geometry,} 1986, 169--177.

\bibitem{ChKNT08}
J. Chun, M.~Korman, M.~Nöllenburg, and T.~Tokuyama: Consistent digital rays. In: \emph{Proc. 22nd Ann. Symposium on Computational Geometry,} 2008, 355--364.

\bibitem{Ci78} P.~G.~Ciarlet: \emph{The \textit{Finite }Element Method for Elliptic Problems}, North-Holland, 1978.

\bibitem{CoR98}
J.~H.~Conway and C.~Radin: Quaquaversal tilings and rotations. \emph{Inventiones Mathematicae} \textbf{132} (1) (1998), 179--188.

\bibitem{deBKV02}
M.~de Berg, M.~J.~Katz, A.~F.~van der Stappen, and J.~M.~Vleugels: Realistic input models for geometric algorithms. \emph{Algorithmica} \textbf{34} (1) (2002), 81--97.

\bibitem{DoFS90}
D.~P.~Dobkin, S.~J.~Friedman, and K.~J. Supowit: Delaunay graphs are almost as good as complete graphs. \emph{Discrete \& Computational Geometry} \textbf{5} (1990), 399--407.

\bibitem{Gr08}
C.~M.~Gray: \emph{Algorithms for Fat Objects: Decompositions and Applications}. Ph.D. thesis, Technische Universiteit Eindhoven, 2008.

\bibitem{KeG92}
J.~M.~Keil and C.~A.~Gutwin: Classes of graphs which approximate the complete Euclidean graph. \emph{Discrete \& Computational Geometry} \textbf{7} (1992), 13--28.

\bibitem{Le59}
K.~Leichtweiss, \"{U}ber die affine Exzentrizit\"at konvexer K\"orper. \emph{Archiv der Mathematik} \textbf{10} (1959), 187--199.

\bibitem{PaPS90}
J.~Pach, R.~Pollack, and J.~Spencer: Graph distance and Euclidean distance on the grid. In: \emph{Topics in Combinatorics and Graph Theory (R.~Bodendiek et al., eds)}, Physica Verlag, Heidelberg, 1990, 555--559.

\bibitem{PeS89}
D.~Peleg and A.~A.~Sch\"affer: Graph spanners. \emph{Journal of Graph Theory} \textbf{13} (1) (1989), 99--116.

\bibitem{Ra94}
C.~Radin: The pinwheel tilings of the plane. \emph{Annals of Mathematics} \textbf{139} (3), (1994), 661--702.

\bibitem{Ra95}
C.~Radin: Symmetry and tilings. \emph{Notices of the AMS} \textbf{42} (1) (1995): 26--31.

\bibitem{RaS96}
C.~Radin and L.~Sadun: The isoperimetric problem for pinwheel tilings. \emph{Communications in Mathematical Physics} \textbf{177} (1) (1996), 255--263.

\bibitem{Sh14}
J.~R.~Shewchuk: \emph{Delaunay Mesh Generation}, CRC Press, 2014.

\bibitem{vdSO94}
A.~v. d.~Stappen and M.~H.~Overmars:
Motion planning amidst fat obstacles. In: \emph{Proc.~10th~ACM~Symposium on Computational Geometry} 1994, 31--40.

\bibitem{vdSHO93}
A.~F.~van der Stappen, D.~Halperin, and M.~H.~Overmars: The complexity of the free space for a robot moving amidst fat obstacles. \emph{Computational Geometry: Theory and Applications} \textbf{3} (6) (1993), 353--373.

\end{thebibliography}
\end{document}